\documentclass{amsart}

\usepackage{amsmath,amsthm,mathrsfs,amssymb,graphicx}
\usepackage[all]{xy}
\usepackage[unicode=true,
					 pdfusetitle=true,
					 bookmarks=true,
					 bookmarksnumbered=false,
					 bookmarksopen=false,
					 breaklinks=true,pdfborder={0 0 0},
					 backref=false,
					 colorlinks=true,
					 plainpages=false,
					 pdfpagelabels,
					 linktocpage=true]{hyperref}

\newtheorem{theorem}{Theorem}[section]
\newtheorem{corollary}[theorem]{Corollary}

\newtheorem{lemma}[theorem]{Lemma}
\newtheorem{proposition}[theorem]{Proposition}

\theoremstyle{definition}

\newtheoremstyle{principle}{}{}{\itshape}{}{\bfseries}{.}{.5em}{\thmnote{#3}#1}
\theoremstyle{principle}
\newtheorem*{principle}{}
\newtheoremstyle{case}{}{}{}{}{\itshape}{.}{.5em}{\thmnote{Case #3}#1}
\theoremstyle{case}

\newcommand{\N}{\mathbb{N}}

\newcommand{\Iff}{\Longleftrightarrow}

\newcommand{\0}{\mathbf{0}}

\begin{document}
\title{Computably enumerable partial orders}
\author[P.~A.~Cholak]{Peter A. Cholak}
\address{Department of Mathematics\\
University of Notre Dame\\
Notre Dame, Indiana 46556 U.S.A.}
\email{cholak@nd.edu}
\author[D.~D.~Dzhafarov]{Damir D. Dzhafarov}
\address{Department of Mathematics\\
University of Notre Dame\\
Notre Dame, Indiana 46556 U.S.A.}
\email{ddzhafar@nd.edu}
\author[N.~Schweber]{Noah Schweber}
\address{Department of Mathematics\\
University of California\\
Berkeley, California 94720 U.S.A.\\
}
\email{schweber@math.berkeley.edu}
\author[R.~A.~Shore]{Richard A. Shore}
\address{Department of Mathematics\\
Cornell University\\
Ithaca, NY 14853 U.S.A.\\
}
\email{shore@math.cornell.edu}
\thanks{The authors are grateful to Julia Knight for helpful comments. This research was partly conducted while Dzhafarov and Shore were visiting the Institute for Mathematical Sciences at the National University of Singapore in 2011, with funding from the John Templeton Foundation. Cholak was partially supported by NSF grant DMS-0800198, Dzhafarov by an NSF Postdoctoral Fellowship, and Shore by NSF grant DMS-0852811.}
\maketitle

\begin{abstract}
We study the degree spectra and reverse-mathematical applications of computably enumerable and co-computably enumerable partial orders. We formulate versions of the chain/antichain principle and ascending/descending sequence principle for such orders, and show that the former is strictly stronger than the latter. We then show that every $\emptyset'$-computable structure (or even just of c.e.\ degree) has the same degree spectrum as some computably enumerable (co-c.e.)\ partial order, and hence that there is a c.e.\ (co-c.e.)\ partial order with spectrum equal to the set of nonzero degrees.
\end{abstract}

\section{Introduction}

A major theme of applied computability theory is the study of the algorithmic properties of countable structures and their presentations, and of the logical content of theorems concerning them. Partial orders, in particular, have been investigated extensively, most recently by Downey, Hirschfeldt, Lempp, and Solomon \cite{DHLS-2003}; Hirschfeldt and Shore \cite{HS-2007}; Jockusch, Kjos-Hanssen, Lempp, Lerman, and Solomon \cite{JKLLS-2009}; and Greenberg, Montalb\'{a}n, and Slaman \cite{GMS-ta}.

There are several approaches taken in such analyses. Most commonly, we restrict attention to computable orders and study the effectivity (or lack thereof) of particular combinatorial constructions or objects of interest. In computable model theory we might consider noncomputable orders, and inquire instead about which ones admit computable (isomorphic) copies, or more generally, in which Turing degrees copies can be found and how complicated the witnessing isomorphisms are. Finally, in reverse mathematics we formalize theorems pertaining to partial orders, and calibrate the strengths of these theorems according to which set-existence axioms are necessary to prove them. There is a fruitful interplay between these approaches, particularly the first and last, with results and insights from one often leading to results in the others. For a detailed account of this relationship, see for example \cite[Section 1]{HS-2007}. We refer the reader to Soare \cite{Soare-1987} for background in computability theory; to Ash and Knight \cite{AK-2000} for computable model theory; and to Simpson \cite{Simpson-2009} for reverse mathematics.

In this article, we look at several questions pertaining to computably enumerable (c.e.)\ and co-computably enumerable (co-c.e.)\ partial orders. Such orders, in which the relationship between a given pair of elements is not decidable but rather may only be revealed over time, arise naturally in different contexts in computability theory. For example, the inclusion order on a uniformly computable family of sets (possibly with repetition) is co-c.e., and in fact, in terms of computational complexity the two are interchangeable, as we show (Proposition \ref{prop_famequiv}). Previous work on c.e.\ partial orders was done by Case \cite{Case-1976}, who studied their extensibility to total orders, and Roy \cite{Roy-1993}, who studied which c.e.\ binary relations have computable copies.

Our interest was motivated by generalizations of the chain/antichain principle ($\mathsf{CAC}$) and the ascending/descending sequence principle ($\mathsf{ADS}$); see Section \ref{sec_ads} for definitions. In the context of reverse mathematics, where the orders one considers may be regarded as being computable, it is easy to see that $\mathsf{CAC}$ implies $\mathsf{ADS}$ over $\mathsf{RCA}_{0}$, but it is unknown whether the reverse implication holds. We show that when the principles are formulated for (formalizations of) c.e.\ and co-c.e.\ partial orders, the answer to the analogous question is no (Proposition \ref{prop_adsequiv} and Corollary \ref{cor_cac}).

We then look at the degree spectra of c.e.\ and co-c.e.\ partial orders on $\omega$, the degree spectrum of an order being the class of degrees containing a copy of it. We show (Theorem \ref{thm_cediff}) that even though there are c.e.\ partial orders with no co-c.e.\ copy, and co-c.e.\ ones with no c.e.\ copy, the degree spectra of the two classes of orders coincide (Corollary \ref{cor_equivsp}). Our main result is that the degree spectra of these partial orders are in fact universal for $\emptyset ^{\prime }$-computable structures: for every such structure (or indeed for any structure of c.e.\ degree) there is a c.e.\ (or co-c.e.)\ partial order with the same degree spectrum, and of the same degree if the structure has c.e.\ degree (Theorem \ref{thm_ceuniv}). As a corollary, we obtain a new example of the Slaman-Wehner theorem (\cite{Slaman-1998}, \cite{Wehner-1998}) that there is a structure with a copy in every nonzero degree. Specifically, it follows that such a structure can be found among the c.e.\ and co-c.e.\ partial orders in any nonzero c.e.\ degree (Corollary \ref{cor_SWce}).

\section{Preliminaries}

\label{sec_prelim}

In this paper, we are interested in countably infinite structures and (unless otherwise noted) assume that their domains are $\omega $. We
also identify orders with the binary relations defining them. Our terminology and notation will for the most part be standard. Given a partial
order $\leq _{P}$ on $\omega $ and $A,B\subseteq \omega$, we say $A \leq_P B$ when $(\forall a\in A)(\forall b\in B)[a\leq _{P}b]$, and similarly for $A <_P B$ and $A \nleq_P B$. When $A$ or $B$ is finite, we often replace it by its members, as in $A\leq_P b$ or $a \leq_P b_0,b_1$.

In referring to or building a c.e.\ partial order $\leq _{P}$ on $\omega $, we usually assume we begin with a computable partial order. We then enumerate additional pairs $\left\langle a,b\right\rangle $ into the order relation and speak of (computably) adding elements to (the graph of) $\leq _{P}$, or of setting $a\leq _{P}b$ for some $a,b\in \omega $. In the co-c.e.\ case, we again begin with some computable partial order, and then remove pairs, or set $a\nleq _{P}b$. In both cases, of course, we must take care to build a transitive relation.

In this section we prove that c.e.\ and co-c.e.\ partial orders cannot be used interchangeably. The following theorem extends Theorem 2.5 of Roy~\cite{Roy-1993} that there is a c.e.\ antisymmetric binary relation with no computable copy.

\begin{theorem}
\label{thm_cediff} There exists a co-c.e.\ partial order on $\omega $ which is not isomorphic to any c.e.\ such order, and conversely.
\end{theorem}

\begin{proof}
We partition $\omega$ computably into the following sets:

\begin{itemize}
\item $\{a, b, c, f, l\}$;

\item $A = \{a_{i,k} : i \in \omega, k < i+1\}$;

\item $B = \{b_i : i \in \omega\}$;

\item $C = \{c_{i,k} : i \in \omega, k < i \}$.
\end{itemize}

We shall use the elements of the first of these sets to identify the elements of the others, and the relations that hold between them, in a given copy of $(\omega ,\leq _{P})$. For each $i\in \omega $, the $a_{i,k}$ for $k<i+1$ will code the number $i$ as a sequence of length $i+1$. The elements $f$ and $l$ will identify the first and last element of this sequence, and the $c_{i,k}$ for $k<i$ will identify pairs of consecutive elements.

Let $U$ be a fixed $\Sigma^0_2$-complete subset of $\omega$, and let $R$ be a computable predicate so that $U = \{i : (\exists x)(\forall y)R(i,x,y) \}$. We use $B$ to represent whether or not $U(i)$ holds, with $b_x \leq_P a_{i,0},\ldots,a_{i,i}$ if and only if we believe that $(\forall y)R(i,x,y)$ holds. Whenever we find some $y$ such that $R(i,x,y)$ does not hold, we thus set $b_x \nleq_P a_{i,0},\ldots,a_{i,i}$, so that in the end we have that $U(i)$ holds if and only if there is an $n \in B$ with $n \leq_P a_{i,0},\ldots,a_{i,i}$, namely $n = b_x$ for the least $x$ such that $(\forall y)[R(i,x,y)]$.

Formally, we build $\leq_P$ as follows. The initial setup is to remove elements from $\omega ^{2}$ so that no relations hold except for those needed for reflexivity and the following:

\begin{itemize}
\item $A <_P a$;

\item $B <_P b$;

\item $C <_P c$;

\item $B <_P A$;
\end{itemize}
and for all $i \in \omega$,

\begin{itemize}
\item $a_{i,0} <_P f$ and $a_{i,i} \leq_P l$;

\item $a_{i,k} <_P c_{i,k}$ and $a_{i,k+1} <_P c_{i,k}$ for all $k < i$.
\end{itemize}
We now proceed by stages, at each of which we remove at most finitely many elements from the graph of $\leq_P$. Associate to each $i \in \omega$ a natural number called its \emph{witness}, initially declared to be $0$. At stage $s$, consider each $i \leq s$ in turn, and suppose the witness of $i$ is $x$. If $R(i,x,y)$ holds for all $y \leq s$, do nothing. Otherwise, set $b_x \nleq_T a_{i,0},\ldots,a_{i,i}$, and redefine the witness of $i$ to be $x+1$.

The resulting order is clearly co-c.e. We also have that $A$ is the set of elements below $a$ and not below $b$, $B$ is the set of elements below $b$, and $C$ the set of elements below $c$. Hence, the images of these sets in any copy of $(\omega,\leq_P)$ are computable in that copy, and if the order in that copy is c.e.\ then so are $B$ and $C$.

Seeking a contradiction, suppose there exists a c.e.\ copy $(\omega,\leq_Q)$ of $(\omega,\leq_P)$. Identify $f$, $l$, $A$, $B$, and $C$ with their images in this copy. Given $i$, we search for elements $a_0,\ldots,a_i \in A$ such that

\begin{itemize}
\item $f \leq_Q a_0$ and $l \leq_Q a_i$;

\item for each $j < i+1$, there is an element of $C$ below both $a_{i,k}$ and $a_{i,k+1}$.
\end{itemize}
This search is computable since $\leq_Q$ and $C$ are c.e. Furthermore, the search must succeed since the images of $a_{i,0},\ldots,a_{i,i}$ would do, and these are also the only possibilities. By construction, we thus have that $i \in U$ if and only if there is an $n \in B$ with $n \leq_Q a_0,\ldots,a_i$. But this now yields a $\Sigma^0_1$ definition of $U$ since $B$ is c.e., contradicting the choice of $U$ as a $\Sigma^0_2$-complete set.

A similar argument yields a c.e.\ partial order with no co-c.e.\ copy.
\end{proof}

While we are here particularly interested in partial orders, there are other natural binary relations one could consider. For example, partial orders can be viewed as directed graphs and we could obtain the preceding theorem for graphs by virtually the same argument. These classes of relations (and others) were studied in \cite{HKSS-2002} alongside partial orders, and the same results were obtained for each. Similarly, we can recast all the results of Section \ref{spectra} about spectra of c.e.\ and co-c.e.\ structures for any of these relational structures (see the comments at the end of Section \ref{spectra}).

\section{Generalizations of $\mathsf{ADS}$ and $\mathsf{CAC}$}\label{sec_ads}

In this section, we investigate the reverse mathematical strength of several mild generalizations of the principles $\mathsf{CAC}$ and $\mathsf{ADS}$, defined below. Recall that a \emph{chain}, respectively, \emph{antichain}, for a partial order $\leq _{P}$ on an arbitrary set $S \subseteq \N$ is a subset of $S$ any two members of which are $\leq _{P}$-comparable, respectively, $\leq_{P} $-incomparable. An \emph{ascending sequence}, respectively \emph{descending sequence}, for $\leq _{P}$, is a subset of $S$ on which $\leq_{P} $ agrees with the natural order $\leq _{N}$, respectively, the reverse natural order $\geq _{N}$.

\begin{principle}[Chain/antichain principle ($\mathsf{CAC}$)]
Every partial order on $\mathbb{N}$ has either an infinite chain or an infinite antichain.
\end{principle}

\begin{principle}[Ascending/descending sequence principle ($\mathsf{ADS}$)]
Every linear order on $\mathbb{N}$ has either an infinite ascending sequence or an infinite descending sequence.
\end{principle}

The computability-theoretic and reverse mathematical content of $\mathsf{CAC}$ and $\mathsf{ADS}$ was studied by Hirschfeldt and Shore \cite{HS-2007}, who established, among other results, that both principles are strictly weaker than Ramsey's theorem for pairs ($\mathsf{RT}^2_2$) and incomparable with $\mathsf{WKL}_0$. While $\mathsf{CAC}$ implies $\mathsf{ADS}$ over $\mathsf{RCA}_0$, it is unknown whether $\mathsf{ADS}$ implies $\mathsf{CAC}$ (\cite{HS-2007}, Question 6.1). We show that the answer to the above question is no when we consider more general types of orders.

To begin, we can easily formalize the notions of c.e.\ and co-c.e.\ partial orders in $\mathsf{RCA}_{0}$: we let a c.e.\ partial order be any function whose values, viewed as pairs, satisfy reflexivity, antisymmetry, and transitivity; similarly for co-c.e.\ orders. This allows us to formulate $\mathsf{CAC}$ and $\mathsf{ADS}$ for c.e.\ and co-c.e.\ partial orders on $\mathbb{N}$. (Instances of these principles are thus no longer analogous to computable partial and linear orders, but rather to c.e.\ or co-c.e.\ ones.)

A different but related class of orders is that of inclusion orders on families of sets. For computable families, inclusion is co-c.e. But since a family of sets can in general contain repetitions, we do not necessarily obtain a co-c.e.\ partial order isomorphic to the inclusion order on a given computable family $\langle A_{i}:i\in \omega \rangle $ simply by setting $i\leq j$ if $A_{i}\subseteq A_{j}$. However, we do obviously obtain a preorder (a reflexive, transitive relation). The converse is also true:

\begin{proposition}
\label{prop_famequiv} Every co-c.e.\ preorder on $\omega$ is isomorphic to the inclusion order on a computable family of sets.
\end{proposition}

\begin{proof}
Fix a co-c.e.\ preorder $(\omega ,\leq _{P})$, along with a computable enumeration of its complement. Write $i\nleq _{P,s}j$ and $i\leq _{P,s}j$ depending on whether the pair $\langle i,j\rangle $ has or has not been enumerated by stage $s$, and assume our enumeration has been speeded up, if necessary, to ensure that the relation $\leq _{P,s}$ is transitive on $\omega \mathbin{\upharpoonright}s+1$. We define $\langle A_{i}:i\in \mathbb{N}\rangle $ by stages, defining the $A_{i}$ for $i\leq s$ on a common initial segment of $\omega $ at stage $s$. The isomorphism will be given by $i\mapsto A_{i}$.

At stage $s$, add to $A_{s}$ all elements of $A_{i}$ for all $i\leq s$ with $i\leq _{P,s}s$. Then, consider any $i,j\leq s$ with $i\nleq _{P,s}j$, and let $n$ be the least number at which no set has yet been defined. For each $k\leq s$, add $n$ to $A_{k}$ if $i\leq _{P,s}k$, and to the complement of $A_{k}$ otherwise.

The family $\langle A_i: i \in \mathbb{N} \rangle$ is computable. It is also clear that if $i \nleq_P j$ then $A_i \nsubseteq A_j$, as some number will be put into $A_i - A_j$ at a sufficiently large stage. Now suppose $i \leq_P j$, so that $i \leq_{P,s} j$ for all $s$. By our assumption about $\leq_{P,s} $, for any stage $s \geq \max\{i,j\}$ and any $k \leq s$, we will necessarily have $k \leq_{P,s} j$ whenever $k \leq_{P,s} i$. Thus, any number put into $A_i$ at such a stage will also belong to $A_j$. Since numbers can only be put into $A_i$ at a stage $s \geq i$, it thus follows that if $j < i$ then $A_i \subseteq A_j$. On the other hand, if $i < j$, then all elements of $A_i$ are put into $A_j$ at stage $s = j$, so we reach the same conclusion.
\end{proof}

We can formalize c.e.\ and co-c.e.\ preorders on $\mathbb{N}$ in $\mathsf{RCA}_{0}$ just as we did partial orders, and then consider formulations of $\mathsf{CAC}$ and $\mathsf{ADS}$ for them as well. We can also look at versions for actual preorders, that is, those given as relations rather than graphs of functions. While we could formulate versions of $\mathsf{CAC}$ and $\mathsf{ADS}$ directly for families of sets under inclusion, this would be unnecessary by the preceding proposition. (Note that the proof of that proposition can be carried out in $\mathsf{RCA}_{0}$.)

Each of the generalized principles can easily be seen to imply the corresponding original one. In the case of $\mathsf{ADS}$, and in the case of $\mathsf{CAC}$ for preorders on $\mathbb{N}$, we also obtain reversals:

\begin{proposition}
\label{prop_adsequiv} Over $\mathsf{RCA}_{0}$, the following are equivalent:

\begin{enumerate}
\item $\mathsf{ADS}$;

\item $\mathsf{ADS}$ for c.e.\ partial orders on $\mathbb{N}$;

\item $\mathsf{ADS}$ for co-c.e.\ partial orders on $\mathbb{N}$;

\item $\mathsf{ADS}$ for c.e.\ preorders on $\mathbb{N}$;

\item $\mathsf{ADS}$ for co-c.e.\ preorders on $\mathbb{N}$;

\item $\mathsf{ADS}$ for preorders on $\mathbb{N}$.
\end{enumerate}

In addition, $\mathsf{CAC}$ for preorders on $\mathbb{N}$ is equivalent to $%
\mathsf{CAC}$.
\end{proposition}

\begin{proof}
That (1) is implied by each of the other statements is clear. To see that (1) implies (2), note that a c.e.\ linear order on $\mathbb{N}$ must, for all distinct pairs $i,j \in \mathbb{N}$, have precisely one of $\langle i,j \rangle$ or $\langle j,i \rangle$ in its range. Thus, its range exists by $\Delta^0_1$ comprehension and is therefore a linear order that we may apply $\mathsf{ADS}$ to. An ascending or descending sequence for this order is such a sequence for the original c.e.\ order. The same argument shows that (1) implies (3).

Each of (4) and (5) implies (6), so it remains only to show that (1) implies (4) and (5). We prove the former, the proof of the latter being similar. So let a c.e.\ linear preorder be given. Formally, this is a function, but we denote it by $\leq_P$ and write $i \leq_P j$ to mean that $\langle i,j
\rangle$ is in the range. We write $i \sim_P j$ if $i \leq_P j$ and $j \leq_P i$. There are two cases to consider.

First, suppose there exist $i_0, \ldots, i_{n-1} \in \omega$ such that $i_j \not\sim_P i_k$ for all distinct $j, k < n$, and for every $i$ there is a $j
< n$ such that $i \sim_P i_j$. Then for each $j < n$, the set $I_j = \{ i \in \mathbb{N}: i \sim_P j \}$ exists by $\Delta^0_1$ comprehension. By the infinitary pigeonhole principle ($\mathsf{RT}^1$), which follows from $\mathsf{ADS}$ (\cite[Proposition 4.5]{HS-2007}), we may fix a $j < n$ such that $I_j$ is infinite. Then this is an ascending (and also descending) sequence for $\leq_P$.

Now suppose no such $i_{0},\ldots ,i_{n-1}$ as above exist. Then for every finite set $F$ there exists an $i>\max F$ such that $i\not\sim _{P}j$ for all $j\in F$, and, in fact, there is a function $g$, total by assumption, that assigns to each $F$ the least such $i$. We now define by primitive recursion a function $f:\mathbb{N}\rightarrow \mathbb{N}$ such that $f(i)=g(f\mathbin{\upharpoonright}i)$ for all $i\in \mathbb{N}$. Let $R$ be the range of $f$, which exists because $f$ is strictly increasing and is thus also infinite. For all $i,j\in R$ we have $i<_{P}j$ or $j<_{P}i$. If we write the elements of $R$ as $r_{0}<r_{1}<\cdots $, we can define a linear order $<_{Q} $ on $\mathbb{N}$ by $i<_{Q}j$ if and only if $r_{i}<_{P}r_{j}$, which exists by $\Delta _{1}^{0}$ comprehension. If $S$ is any infinite ascending or descending sequence for $(\mathbb{N},\leq _{Q})$ as given by $\mathsf{ADS}$, then $\{r_{i}:i\in S\}$ is such a sequence for $(\mathbb{N},\leq _{P})$.

A similar argument establishes the equivalence of $\mathsf{CAC}$ for preorders on $\mathbb{N}$ with $\mathsf{CAC}$.
\end{proof}

In contrast to the preceding theorem it follows by Theorem \ref{thm_cediff} that there are c.e.\ and co-c.e.\ partial orders with no computable copies. This also follows from the next theorem, as computable instances of $\mathsf{CAC}$ always have solutions that do not compute $\emptyset ^{\prime }$ (see Diagram~1 in \cite{HS-2007}).

\begin{theorem}
\ 

\begin{enumerate}
\item There exists a co-c.e.\ partial order on $\omega $ with no infinite antichains and with all chains computing $\emptyset ^{\prime }$.

\item There exists a c.e.\ partial order on $\omega $ with no infinite chains and with all infinite antichains computing $\emptyset ^{\prime }$.
\end{enumerate}
\end{theorem}

\begin{proof}
To prove (1), fix a computable enumeration $\langle \emptyset _{s}^{\prime}:s\in \omega \rangle $ of $\emptyset ^{\prime }$, and let $t_{i}$ be the least $s$ such that $\emptyset _{s}^{\prime }\mathbin{\upharpoonright}i=\emptyset ^{\prime }\mathbin{\upharpoonright}i$. We build a co-c.e.\ partial order $\leq _{P}$ by stages. To begin, we make $\leq _{P}$ agree with the natural order on $\omega $.

At stage $0$, we do nothing. At stage $s>0$, we consider consecutive substages $i\leq s$. At substage $i$, if no number enters $\emptyset^{\prime }\mathbin{\upharpoonright}i$ at stage $s$, we do nothing and go either to substage $i+1$ or to stage $s+1$, depending on whether $i<s$ or $i=s$. Otherwise, for all $j,k$ with $i\leq j<k\leq s$, we make $j\nleq _{P}k$ and go to stage $s+1$.

It is easily seen that $\leq _{P}$ is a co-c.e., reflexive, antisymmetric relation, while its transitivity follows from the fact that $i\leq _{P}j$ implies $i\leq j$, and that if $i\nleq _{P}k$ for some $k>i$ then necessarily $i\nleq _{P}j$ for all $j$ with $i<j\leq k$. So $\leq _{P}$ is a partial order, and we note that it admits no infinite antichain since for every $i$, $i\leq _{P}j$ for all $j\geq t_{i}$.

Suppose now that $C = \{c_0 < c_1 < \cdots \}$ is an infinite chain for $\leq_P$, so that $c_0 <_P c_1 <_P \cdots$. In particular, since $i \leq c_i$, we must have $c_{i+1} \geq t_i$ for all $i$ by construction. Thus for every $i$ we have $\emptyset^{\prime}\mathbin{\upharpoonright} i =\emptyset^{\prime}_{c_{i+1}} \mathbin{\upharpoonright} i$, whence it follows that $\emptyset^{\prime}\leq_T C$, as desired.

The argument for (2) is analogous.
\end{proof}

Formalizing the previous argument immediately yields the following (note that (4) and (5) are included since all partial orders are preorders):

\begin{corollary}
\label{cor_cac}Over $\mathsf{RCA}_{0}$, the following are equivalent:

\begin{enumerate}
\item $\mathsf{ACA}_0$;

\item $\mathsf{CAC}$ for c.e.\ partial orders on $\mathbb{N}$;

\item $\mathsf{CAC}$ for co-c.e.\ partial orders on $\mathbb{N}$;

\item $\mathsf{CAC}$ for c.e.\ preorders on $\mathbb{N}$;

\item $\mathsf{CAC}$ for co-c.e.\ preorders on $\mathbb{N}$.
\end{enumerate}
\end{corollary}

\noindent Together with Proposition \ref{prop_adsequiv}, this provides a separation between the c.e.\ and co-c.e.\ forms of $\mathsf{CAC}$ on the one hand, and the c.e.\ and co-c.e.\ forms of $\mathsf{ADS}$ on the other.

\section{Degree spectra}\label{spectra}

The \emph{degree spectrum} of a countable structure $\mathcal{S}$ is the set of Turing degrees of copies of $\mathcal{S}$. The study of degree spectra, and in particular, of which classes of degrees can be realized as spectra, has been the subject of many investigations in computable model theory. A partial survey of previous results appears in Section 1 of \cite{HKSS-2002}.

Every structure $\mathcal{S}$ we consider below will be assumed to be in a computable language, with computable signature. We also assume that $
\mathcal{S}$ is automorphically nontrivial (meaning there is no finite $F\subseteq |\mathcal{S}|$ such that each permutation of $|\mathcal{S}|$ fixing $F$ is an automorphism), as trivial structures are structurally uninteresting. By Knight's theorem (\cite{Knight-1986}, Theorem 4.1), the degree spectrum of any $\mathcal{S}$ we consider is therefore closed upwards, and thus also equal to the set of degrees of sets that compute a copy of $\mathcal{S}$.

In this section, we show that the degree spectra of c.e.\ and co-c.e.\ partial orders are universal for $\emptyset ^{\prime }$-computable structures. In contrast with Theorem \ref{thm_cediff}, this also shows that the classes of degree spectra of c.e.\ and co-c.e.\ partial orders are the same.

We begin with a reduction of arbitrary structures to graphs (symmetric, irreflexive binary relations). As with orders, we identify graphs on $\omega$ with their relations.

\begin{lemma}
\label{lem_binrel} For every structure $\mathcal{S}$ with domain $\omega$ there exists a graph $R$ on $\omega $ such that $R \equiv _{T}\mathcal{S}$ and $R$ and $\mathcal{S}$ have the same degree spectrum.
\end{lemma}

\begin{proof}
This proof of this result is folklore, and has appeared in various forms in the published and unpublished literature. A complete proof can be found in \cite[Appendix A]{HKSS-2002}.
\end{proof}

\begin{lemma}
\label{lem_ceuniv} For every $\emptyset ^{\prime }$-computable graph $R$ on $\omega $ there exists a c.e.\ (and a co-c.e.)\ partial order $\leq _{P}$ on $\omega $ with the same degree spectrum. Furthermore, if $R$ has c.e.\ degree then $\deg (\leq _{P})=\deg (R)$.
\end{lemma}

\begin{proof}
We begin by providing a uniform effective transformation taking any graph $R$ to a partial order $\leq _{P}$ with a uniformly effective inverse transformation taking any copy of this partial order to a copy of $R$. Thus $R$ and $\leq _{P}$ have the same degree spectrum. We partition $\omega $ computably into the following sets:

\begin{itemize}
\item $\{a,g,r_{0},r_{1}\}$;

\item $A=\{a_{i}:i\in \omega \}$

\item $G=\{g_{i,j,k}:i<j,k\in \omega \}$.
\end{itemize}
Intuitively, the $a_{i}$ represent the elements of the domain of a given copy of $R$ and the $g_{i,j,k}$ represent guesses at whether or not $R$ holds of the pair $(a_{i},a_{j})$ in that copy. We code these guesses using $r_{0}$ and $r_{1}$, with $r_{0}$ representing that $R$ does not hold, and $r_{1}$ that it does. (The guessing will be tied to a computable approximation of $R$ when $\deg(R)$ is c.e.) The numbers $a$ and $g$ help us recognize the pieces of the above partition in a given copy of $\leq _{P}$.

We can describe the isomorphism type of the desired partial order by saying that $A$ is the set of elements below $a$ and not below $g$, and $G$ is the set of elements below $a$ and $g$. For each $i<j$ and all $k$, $g_{i,j,k}$ is below $a_i$, $a_j$, and $g$. Additionally, all but one $g_{i,j,k}$ are below both $r_0$ and $r_1$, and exactly one is below just one of $r_0$ and $r_1$. The latter $g_{i,j,k}$ is below $r_0$ if $R(i,j)$ does not hold, and below $r_1$ otherwise. It is easy to see that we can build an order of this type computably in $R$, and that, conversely, any order of this type computes a copy of $R$ (and indeed, that the transformation in either direction preserves degree). The point here is that, given a copy of this partial order, we start with the elements corresponding to $a,g,r_{0}$ and $r_{1}$ and then compute a copy of $R$ with domain $A$ by, for each $a_0,a_1\in A$, searching for a $g\in G$ with $g\leq a_0,a_1$ such that $g \nleq r_{0}$ or $g\nleq r_{1}$. Exactly one of these must happen and we let there be an edge between $a_0$ and $a_1$ if and only if it is the first.

When $R$ is $\emptyset'$-computable, we make $\leq_P$ c.e.\ as follows. By Knight's theorem, as discussed above, we may assume that $R$ has c.e.\ degree, since if not we can replace it by a copy of degree $\0'$. Initially, we let $\leq_P$ be the transitive closure of the following conditions:

\begin{itemize}
\item $A <_P a$;

\item $G <_P g$;

\item $g_{i,j,k} \leq_P a_i,a_j$ for all $i < j$ and $k$.
\end{itemize}
By the modulus lemma there is a computable function $f(i,j,s)$ with $\lim_s f(i,j,s) = 0$ if $R(i,j)$ does not hold and $\lim_s f(i,j,s) = 1$ if it does, and with modulus of convergence $m$ of the same degree as $R$. At stage $s > 0$, for every $i < j \leq s$, we let $k$ be the largest number less than or equal to $s$ such that $f(i,j,k) \neq f(i,j,k-1)$, or $0$ if there is no such number. We make $g_{i,j,k} \leq_P r_{f(i,j,k)}$ and $g_{i,j,l} \leq_P r_0,r_1$ for all $l \leq s$ not equal to $k$.

It is easy to see from the limit properties of $f(i,j,s)$ that $\leq_P$ is c.e.\ and that it has the isomorphism type described above. Thus, in particular, its degree spectrum is the same as that of $R$. Furthermore, for each $i < j$ it is precisely $g_{i,j,m(i,j)}$ which is below just one of $r_{0}$ and $r_{1}$ (indeed, it is below just $r_{f(i,j,m(i,j))} = r_{\lim_s f(i,j,s)}$, as desired). Thus, $\leq_P$ is of the same degree as $m$, and so of the same degree as $R$.

To obtain instead a co-c.e.\ partial order, we modify the construction of $\leq _{P}$ by initially setting $g_{i,j,k}\leq_{P}r_{0}$ and $g_{i,j,k}\leq _{P}r_{1}$ for all $i<j,k\in \omega $, and then removing one of these relations when our guess about $R(i,j)$ changes (including by default at $k=0$) and both of them when it does not. In the end, there will be exactly one $k$ such that $g_{i,j,k}$ is below $r_{0}$ or $r_{1}$. It will be below just one of them and $R(i,j)$ will hold just in case it is below $r_{1}$. This completes the proof.
\end{proof}

\begin{theorem}
\label{thm_ceuniv} For every $\emptyset ^{\prime }$-computable structure $\mathcal{S}$ on $\omega$ there exists a c.e.\ (and a co-c.e.)\ partial order on $\omega $ with the same degree spectrum as $\mathcal{S}$. Furthermore, there exists such a partial order in every c.e.\ degree containing a copy of $\mathcal{S}$
.
\end{theorem}

\begin{proof}
By Lemma \ref{lem_binrel}, for every $\emptyset ^{\prime }$-computable copy of $\mathcal{S}$ on $\omega $ there is a graph $R$ on $\omega$ of the same degree and with the same degree spectrum. By Lemma \ref{lem_ceuniv}, there exists a c.e.\ (and a co-c.e.)\ partial order with the same degree spectrum as $R$, and if the degree of a given copy of $\mathcal{S}$, and hence of the corresponding one of $R$, is c.e., we may choose the order to be of that degree as well.
\end{proof}

As an immediately corollary, we have:

\begin{corollary}
\label{cor_equivsp}Let $\mathcal{S}$ be a structure with domain $\omega $. The following are equivalent:

\begin{enumerate}
\item $\mathcal{S}$ has a copy of c.e.\ degree;

\item $\mathcal{S}$ has a $\emptyset^{\prime}$-computable copy;

\item there is a c.e.\ (and a co-c.e.)\ partial order on $\omega $ with the same degree spectrum as $\mathcal{S}$.
\end{enumerate}
\end{corollary}

By Theorem \ref{thm_ceuniv}, we obtain a wide array of classes of degrees realized as degree spectra of c.e.\ (or co-c.e.)\ partial orders, including, for example, the class of degrees $\geq \mathbf{0}^{\prime }$. A particularly interesting class is that of the degrees strictly above $\mathbf{0}$. Recall the following well-known result:

\begin{theorem}[Slaman \protect\cite{Slaman-1998}, Wehner \protect\cite{Wehner-1998}]\label{thm_sw}
There exists a countable structure whose degree spectrum consists precisely of the nonzero degrees.
\end{theorem}

\noindent Since any two copies of the same structure have the same degree spectrum, it follows that an $\mathcal{S}$ as above can be found in any nonzero degree. The following consequence of Theorem \ref{thm_ceuniv} is a considerably more effective version.

\begin{corollary}
\label{cor_SWce}Every nonzero c.e.\ degree contains a c.e.\ (and a co-c.e.) partial order on $\omega $ whose degree spectrum consists precisely of the nonzero degrees.
\end{corollary}

\begin{proof}
Fix a nonzero c.e.\ degree, and any structure in that degree satisfying Theorem \ref{thm_sw}. Then apply Theorem \ref{thm_ceuniv}.
\end{proof}

As partial orders are directed graphs, all the results of this section hold for directed graphs as well. As for graphs, one could prove a general coding result taking c.e.\ or co-c.e.\ partial orders to c.e.~ or co-c.e.\ graphs as in \cite[Appendix A]{HKSS-2002}. There is a simpler route in our case, however. The crucial item here is Lemma \ref{lem_ceuniv}. Note first that the comparability graph $\widetilde{R}$ of a partial order $\leq _{P}$, defined for $x,y$ in the domain of $\leq_P$ by $\widetilde{R}(x,y) \Iff x <_{P} y~\vee ~y < _{P} x$, is always computable from $\leq _{P}$ and is c.e.\ (co-c.e.)\ if $\leq _{P}$ is. Next, given a graph $R$ as in Lemma \ref{lem_ceuniv} and the corresponding  partial order $\leq _{P}$ computing $R$ constructed there, it is clear that the associated comparability graph $\widetilde{R}$ has enough information to compute $R$, and hence also $\leq _{P}$. It is then a c.e.\ (co-c.e.)\ graph of the same degree, and with the same degree spectrum, as $\leq_P$. Thus all the results of this section also hold for graphs in place of partial orderings.

\bibliographystyle{plain}
\bibliography{ceorderings}

\end{document}